\documentclass[psamsfonts,reqno]{amsart}

\usepackage{amssymb,amsfonts}
\usepackage[all,arc]{xy}
\usepackage{enumerate}
\usepackage{mathrsfs}
\usepackage{hyperref}
\usepackage{mathtools}
\usepackage{comment}
\hypersetup{
	pdftitle={Quantitative Estimates of the Field Excited by an Emitter in the Narrow Region Between Two Strictly Convex Inclusions in Two Dimensions},
	pdfauthor={Saibo Chen},
	pdfsubject={Two-dimensional perfect conductivity problem with a dipole source between closely spaced strictly convex inclusions},
	pdfkeywords={perfect conductivity problem, dipole source, field concentration, harmonic barrier},
	pdfdisplaydoctitle=true
}
\hypersetup{pdfborder={0 0 0}}
\usepackage{bm}
\usepackage{tikz}
\usepackage{pgfplots}
\usetikzlibrary{arrows.meta}

\newtheorem{thm}{Theorem}[section]

\newtheorem{prop}{Proposition}[section]
\newtheorem{lem}{Lemma}[section]
\newtheorem{rk}{Remark}[section]

\theoremstyle{definition}

\theoremstyle{remark}

\makeatletter
\let\c@equation\c@thm
\makeatother
\numberwithin{equation}{section}

\title[Estimates of the Field Excited by an Emitter in Two Dimensions]{Quantitative Estimates of the Field Excited by an Emitter in the Narrow Region Between Two Strictly Convex Inclusions in Two Dimensions}

\author{Saibo Chen}
\address{School of Mathematical Sciences, Fudan University, Shanghai 200433, China}
\email{sbchen25@m.fudan.edu.cn}

\begin{document}
	
	\begin{abstract}
		We consider the field enhancement due to the presence of an emitter of the dipole type near two perfectly conducting inclusions with $C^{2,\gamma}$ boundaries in two dimensions. We derive estimates of field enhancement in the narrow region and prove that the field is enhanced by a factor $\varepsilon^{-1/2}$ in the region far from the emitter. We construct a harmonic barrier to control the dipole term instead of using the Kelvin transformation. We also study how the location of the emitter affects the field on the shortest line segment connecting two inclusions.
	\end{abstract}

	\maketitle
	%\tableofcontents
	\section{Introduction}
	Field concentration in composite materials has been extensively studied in the last two decades. For conductivities bounded away from zero and infinity, uniform gradient estimates were established for scalar equations in \cite{generalk} and extended to elliptic systems in \cite{generalsys}. For perfect conductors, sharp estimates were first obtained for circular inclusions in \cite{Ammari2,Ammari1} and later extended to general shapes and higher dimensions in \cite{perfect-alldim,perfect-dim2-general}. The characterizations of the singular term were developed in \cite{Ammariqqb,KLYP}. Estimates for insulated inclusions were obtained in \cite{insulated-dim2,insulated-alldim general,insulated-alldim,insulated-Harnack,maxprin,Yun Ball}.
	
	Instead of considering field concentration generated by a prescribed background field, a different problem arises when the background field is excited by an emitter in the narrow region between two inclusions. Such a setting is motivated by \cite{phy-emitter}. Kang and Yun studied this problem for the cases when inclusions are of circular shape \cite{emitter1} or of bow-tie shape \cite{emitter2} in two dimensions and of spherical shape in three dimensions \cite{emitter3}.
	
	Before we state our main results, we fix our domain and notation. Let $D_1^0,\ D_2^0\subset\mathbb{R}^2$ be two bounded strictly convex domains with $C^{2,\gamma}(0<\gamma<1)$ boundaries touching at the origin. After a rigid motion, we can suppose  that the common tangent at the origin is the $x_2$-axis, with $D_1^0$ lying to its left and $D_2^0$ to its right. Set
	$$D_1\coloneqq D_1^0-(\varepsilon/2,0),\ \ \ \ D_2\coloneqq D_2^0+(\varepsilon/2,0)$$
	and define
	$$\Omega\coloneqq\mathbb{R}^2\setminus\overline{D_1\cup D_2}.$$
	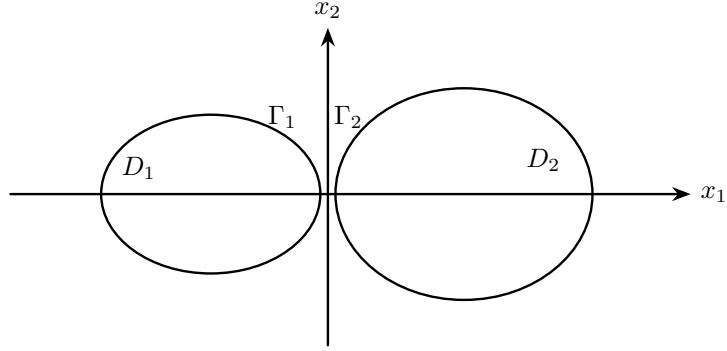
\begin{figure}[htbp]
		\centering
		\begin{tikzpicture}[
			>=Stealth,
			line width=0.9pt,
			line cap=round,
			line join=round
			]
			\draw[->] (-4.2,0) -- (4.8,0) node[right] {$x_1$};
			\draw[->] (0,-2.0) -- (0,2.2) node[above] {$x_2$};
			\draw (-1.55,0) ellipse [x radius=1.45, y radius=1.05];
			\draw (1.80,0) ellipse [x radius=1.70, y radius=1.40];
			\node at (-2.50,0.35) {$D_1$};
			\node at ( 2.85,0.45) {$D_2$};
			\node at (-0.60,1.02) {$\Gamma_1$};
			\node at ( 0.27,1.02) {$\Gamma_2$};			
		\end{tikzpicture}
		\caption{Two strictly convex inclusions $D_1$ and $D_2$.}
	\end{figure}

	The parts of $\partial D_1$ and $\partial D_2$, denoted by $\Gamma_1$ and $\Gamma_2$, are respectively the graphs of two $C^{2,\gamma}$ functions in terms of $x_2$. Precisely, for some $R_0>0$ independent of $\varepsilon$,
	$$\Gamma_1=\{(x_1,x_2):x_1=-\varepsilon/2+f(x_2),\ |x_2|<2R_0\},$$
	$$\Gamma_2=\{(x_1,x_2):x_1=\varepsilon/2+g(x_2),\ |x_2|<2R_0\},$$
	where $f$ and $g$ are $C^{2,\gamma}$ functions satisfying
	$$f(x_2)=-\frac{\kappa_1}{2}x_2^2+O(|x_2|^{2+\gamma}),\ \ \ \  g(x_2)=\frac{\kappa_2}{2}x_2^2+O(|x_2|^{2+\gamma}).$$
	We further assume that the curvatures of the inclusions at the closest points are strictly positive, which means
	$$\kappa_1=-f''(0)>0,\ \ \ \ \kappa_2=g''(0)>0.$$
	Define the narrow region
	$$\Omega_0\coloneqq\{(x_1,x_2):-\varepsilon/2+f(x_2)<x_1<\varepsilon/2+g(x_2),\ |x_2|<R_0\}.$$
	
	Following the mathematical model introduced by Kang and Yun \cite{emitter1}, we consider the perfect conductivity problem with a dipole emitter:
	\begin{equation}\label{equation of u}
		\begin{cases}
			\Delta u=\mathbf{a}\cdot\nabla\delta_{\mathbf{p}}\ \ {\rm in}\ \Omega,\\
			u=u_i\ \ {\rm on}\ \partial D_i,\ \ i=1,2,\\
			\displaystyle\int_{\partial D_i}\partial_\nu u\ {\rm d}s=0,\ \ i=1,2,\\
			u(\mathbf{x})=O(|\mathbf{x}|^{-1})\ \ {\rm as}\ |\mathbf{x}|\rightarrow\infty.
		\end{cases}
	\end{equation}
	Throughout the paper, $\nu$ denotes the outward unit normal to $\Omega$. The distribution $\mathbf{a}\cdot\nabla\delta_{\mathbf{p}}$ represents the emitter of the dipole type, where the unit vector $\mathbf{a}$ is the direction of the dipole and $\mathbf{p}=(s,t)\in\Omega_0$ its location.
	
	Let $\mathcal{N}_{\mathbf{p}}$ be the fundamental solution to the Laplacian in two dimensions, namely,
	$$\mathcal{N}_{\mathbf{p}}(\mathbf{x})\coloneqq\frac{1}{2\pi}\log{|\mathbf{x-p}|}.$$
	In the absence of inclusions, the solution of $\Delta u=\mathbf{a}\cdot\nabla\delta_{\mathbf{p}}$ is given by
	$$u(\mathbf{x})=\mathbf{a}\cdot\nabla\mathcal{N}_{\mathbf{p}}(\mathbf{x}),$$
	whose gradient has a singularity at $\mathbf{p}$ of size $|\mathbf{x-p}|^{-2}$. For two equal circular inclusions with the emitter on the $x_2$-axis, Kang and Yun \cite{emitter1} showed that the field is enhanced at points away from the emitter, and that the enhancement  factor is $\varepsilon^{-1/2}$. These results rely strongly on the geometry of circles, so Kang and Yun point out that it is interesting to extend the results to general strictly convex inclusions.
	
	In this paper, we consider the problem that arises in \cite{emitter1} and get the same enhancement factor. Similarly to \cite{emitter1}, we decompose the solution into a potential-gap term and a dipole term. The estimate for the potential-gap term is standard. As for the dipole term, we construct a local harmonic barrier to control it. This method does not use the geometry of circles. We cannot guarantee that our method applies to the three-dimensional case, because some key lemmas no longer hold, such as Lemmas 3.2 and 3.3.
	
	Throughout the paper, $D_1^0$ and $D_2^0$ are fixed. The constant $R_0$ may be decreased if necessary. Unless otherwise stated, all constants are uniform, which means they are independent of $\varepsilon,\ \mathbf{x},\ \mathbf{a},\ \mathbf{p}$  but may depend on the fixed domains and parameters. The letter $C$ denotes a uniform constant and may change from line to line. Named constants such as $C_1,\ C_2$ are fixed within each result, but may take different values in different results. We write $X\lesssim Y$ if $X\le CY$, and $X\sim Y$ if $C^{-1}Y\le X\le CY$. Let $B_{\delta}(\mathbf{x})$ denote the open disk of radius $\delta$ centered at $\mathbf{x}$, and write $B_{\delta}\coloneqq B_{\delta}(\mathbf{0})$.
	
	Our first theorem concerns dipole directions with a nonzero horizontal component.
	\setcounter{thm}{0}
	\begin{thm}
		Let $u$ be the solution to (\ref{equation of u}). Fix $M>0$, $a_0\in(0,1)$ and $\theta\in(0,1)$. There exist positive constants $\varepsilon_0,\ C_1,\ C_2,\ A,\ c_0$ such that, for every $0<\varepsilon\leq\varepsilon_0$, every unit vector $\mathbf{a}=(a_1,a_2)$ satisfying $|a_1|\ge a_0$, every $\mathbf{p}=(s,t)\in\Omega_0$ satisfying $|t|\le M\sqrt{\varepsilon}$ and ${\rm dist}(\mathbf{p},\partial\Omega)\ge\theta\varepsilon$, and every $\mathbf{x}\in\Omega_0$, the following estimates hold:
		\setcounter{thm}{1}
		\begin{enumerate}
			\item[\textup{(i)}]
			\textbf{Near the emitter.}
			If $0<|\mathbf{x-p}|\le C_1\varepsilon$, then
			\begin{equation}\label{thm1}
				|\nabla u(\mathbf{x})|\sim\frac{1}{|\mathbf{x-p}|^2}.
			\end{equation}
			\item[\textup{(ii)}]
			\textbf{Far from the emitter.}
			If $|\mathbf{x-p}|\ge C_2\varepsilon|\log{\varepsilon}|$, then for $|x_2|\le c_0\varepsilon^{1/4}$ we have
			\begin{equation}\label{thm2}
				|\nabla u(\mathbf{x})|\sim\frac{|a_1|}{\sqrt{\varepsilon}(\varepsilon+x_2^2)},
			\end{equation}
			while for $|x_2|>c_0\varepsilon^{1/4}$ we have
			\begin{equation}\label{thm3}
				|\nabla u(\mathbf{x})|\lesssim\frac{1}{\sqrt{\varepsilon}(\varepsilon+x_2^2)}.
			\end{equation}
			\item[\textup{(iii)}]
			\textbf{Transition region.}
			If $C_1\varepsilon<|\mathbf{x-p}|<C_2\varepsilon|\log{\varepsilon}|$, then
			\begin{equation}\label{thm4}
				|\nabla u(\mathbf{x})|\lesssim\frac{1}{|\mathbf{x-p}|^2}\exp{\left(-\frac{A|\mathbf{x-p}|}{\varepsilon}\right)}+\varepsilon^{-3/2}.
			\end{equation}
		\end{enumerate}
	\end{thm}
	\begin{rk}
		Estimate (\ref{thm2}) shows the enhancement factor is $\varepsilon^{-1/2}$, which is the same as in \cite{emitter1}. Indeed, suppose $2M\sqrt{\varepsilon}\le|x_2|\le 3M\sqrt{\varepsilon}$ (the constant $3M$ is not essential). Since $|t|\le M\sqrt{\varepsilon}$, we have
		$$|\mathbf{x-p}|^2\lesssim(\varepsilon+x_2^2)^2+(\varepsilon+t^2)^2+x_2^2+t^2\lesssim\varepsilon+x_2^2$$
		and
		$$\varepsilon+x_2^2\lesssim\varepsilon\lesssim (x_2-t)^2\lesssim|\mathbf{x-p}|^2.$$
		Therefore,
		$$|\mathbf{x-p}|^2\sim\varepsilon+x_2^2.$$
		In the absence of inclusions, the field satisfies
		$$|\nabla(\mathbf{a}\cdot\nabla\mathcal{N}_{\mathbf{p}})(\mathbf{x})|\sim\frac{1}{\varepsilon+x_2^2},$$
		whereas (\ref{thm2}) gives
		$$|\nabla u(\mathbf{x})|\sim\varepsilon^{-1/2}|\nabla(\mathbf{a}\cdot\nabla\mathcal{N}_{\mathbf{p}})(\mathbf{x})|.$$
	\end{rk}
	We next consider the dipole with vertical direction $\mathbf{a}=(0,1)$. The case $\mathbf{a}=(0,-1)$ follows by linearity.
	\setcounter{thm}{1}
	\begin{thm}
		\setcounter{thm}{5}
		Let $u$ be the solution to (\ref{equation of u}). Fix $M>0,\ \theta\in(0,1)$ and let $\mathbf{a}=(0,1)$. There exist positive constants $\varepsilon_0,\ C_1,\ C_2,\ A$ such that, for every $0<\varepsilon\le\varepsilon_0$, every $\mathbf{p}=(s,t)\in\Omega_0$ satisfying $|t|\le M\sqrt{\varepsilon}$ and ${\rm dist}(\mathbf{p},\partial\Omega)\ge\theta\varepsilon$, and every $\mathbf{x}\in\Omega_0$, the following estimates hold:
		\begin{enumerate}
			\item[\textup{(i)}]
			\textbf{Near the emitter.}
			If $0<|\mathbf{x-p}|\le C_1\varepsilon$, then
			\begin{equation}\label{thm11}
				|\nabla u(\mathbf{x})|\sim\frac{1}{|\mathbf{x-p}|^2}.
			\end{equation}
			\item[\textup{(ii)}]
			\textbf{Far from the emitter.}
			If $|\mathbf{x-p}|\ge C_2\varepsilon|\log{\varepsilon}|$, then
			\begin{equation}\label{thm22}
				|\nabla u(\mathbf{x})|\lesssim\frac{1}{\varepsilon+x_2^2}.
			\end{equation}
			\item[\textup{(iii)}]
			\textbf{Transition region.}
			If $C_1\varepsilon<|\mathbf{x-p}|<C_2\varepsilon|\log{\varepsilon}|$, then
			\begin{equation}\label{thm33}
				|\nabla u(\mathbf{x})|\lesssim\frac{1}{|\mathbf{x-p}|^2}\exp{\left(-\frac{A|\mathbf{x-p}|}{\varepsilon}\right)}+\varepsilon^{-1}.
			\end{equation}
		\end{enumerate}
	\end{thm}
	\begin{rk}
		The estimates in Theorem 1.2 can be seen as improvements over Theorem 1.1 due to the zero horizontal component of $\mathbf{a}$, but they are still weaker than the results in \cite{emitter1} because of the lack of symmetry. When $D_1\cup D_2$ is symmetric with respect to the $x_2$-axis and $\mathbf{p}=(0,t)\in\Omega_0$, we can get similar results to those in \cite{emitter1}. See Remark 4.1.
	\end{rk}
	Theorems 1.1 and 1.2 concern emitters satisfying $|p|\le M\sqrt{\varepsilon}$. We next consider emitters that are farther from the origin. Let
	$$L\coloneqq\{(x_1,0):|x_1|<\varepsilon/2\}$$
	be the shortest line segment connecting two inclusions. We quantify how the location of the emitter affects the field on $L$.
	\setcounter{thm}{2}
	\begin{thm}
		\setcounter{thm}{8}
		Let $u$ be the solution to (\ref{equation of u}), and fix $a_0\in(0,1)$. There exist positive constants $\varepsilon_0,\ A$ and $t_0$ with $t_0<R_0/4$ such that, for every $0<\varepsilon\le\varepsilon_0$, every unit vector $\mathbf{a}=(a_1,a_2)$ satisfying $|a_1|\ge a_0$ and every $\mathbf p=(s,t)\in\Omega_0$ satisfying $4\sqrt{\varepsilon}<|t|\le t_0$, we have
		\begin{equation}\label{location estimate1}
			|\nabla u(\mathbf{x})|\lesssim\frac{1}{\sqrt{\varepsilon}(\varepsilon+t^2)}+\frac{1}{\varepsilon{\rm dist}(\mathbf{p},\partial\Omega)}\exp{\left(-\frac{A}{\sqrt{\varepsilon}}\right)},\ \mathbf{x}\in L.
		\end{equation}
		In particular, for every fixed $\theta\in(0,1)$, if ${\rm dist}(\mathbf{p},\partial\Omega)\ge\theta\varepsilon$, we have
		\begin{equation}\label{location estimate2}
			|\nabla u(\mathbf{x})|\lesssim\frac{1}{\sqrt{\varepsilon}(\varepsilon+t^2)},\ \mathbf{x}\in L.
		\end{equation}
	\end{thm}
	\begin{rk}
		If $\mathbf{a}=(0,1)$, we can also get the corresponding estimates. We omit these results since the proof is the same as that of Theorem 1.3. Consider the case ${\rm dist}(\mathbf{p},\partial\Omega)\ge\theta\varepsilon$ and $|a_1|\ge a_0$. If the emitter is located away from the origin such that $|t|$ is bounded below independently of $\varepsilon$, then by Theorem 1.3, for $\mathbf{x}\in L$ we have $|\nabla u|\lesssim\varepsilon^{-1/2}$. This is the same order as in the perfect conductivity problem without an emitter. If $|t|\sim\sqrt{\varepsilon}$, Theorem 1.3 implies $|\nabla u|\lesssim\varepsilon^{-3/2}$, which is consistent with (\ref{thm2}). Moreover, when $D_1$ and $D_2$ are disks of the same radius, the upper bound in (\ref{location estimate2}) is sharp. See Remark 5.1.
	\end{rk}
	
	The rest of this paper is organized as follows. In Section 2, we introduce auxiliary functions and decompose the solution into two components. In Section 3, we derive estimates for these components. In Section 4, we prove Theorems 1.1 and 1.2. In Section 5, we prove Theorem 1.3.

	\section{Auxiliary Functions and Decomposition of the Solution}
	The decomposition used in this section follows the framework introduced by Kang and Yun \cite{emitter1} for two circular inclusions.
	
	Let $q$ be the solution of the following problem:
	\begin{equation}\label{equation of q}
		\begin{cases}
			\Delta q=0\ \ {\rm in}\ \Omega,\\
			q=q_i\ \ {\rm on}\ \partial D_i,\ \ i=1,2,\\
			\displaystyle\int_{\partial D_i}\partial_\nu q\ {\rm d}s=(-1)^i,\ \ i=1,2,\\
			q(\mathbf{x})=O(|\mathbf{x}|^{-1})\ \ {\rm as}\ |\mathbf{x}|\rightarrow\infty.
		\end{cases}
	\end{equation}
	Here, $q_1$ and $q_2$ are constants determined by the third condition in (\ref{equation of q}) and depend on $\varepsilon$. The existence and uniqueness of the solution to (\ref{equation of q}) were established in \cite{Ammariqqb}.
	
	Let $z_1\in\partial D_1$ and $z_2\in\partial D_2$ be the closest points, namely, ${\rm dist}(D_1,D_2)=|z_1-z_2|$. Let $B_i$ be the osculating disk of $D_i$ at $z_i,i=1,2$, and let $q_B$ be the solution of (\ref{equation of q}) when the two inclusions are $B_1$ and $B_2$. Let $R_i$ denote the inversion with respect to $\partial B_i$. Then $R_1R_2$ and $R_2R_1$ have fixed points $\mathbf{p}_1\in B_1$ and $\mathbf{p}_2\in B_2$. It was proved in \cite{KLYP} that
	\begin{equation}\label{p1p2}
		\mathbf{p}_i=\left((-1)^i\displaystyle\sqrt{\frac{2r_1r_2}{r_1+r_2}}\sqrt{\varepsilon}+O(\varepsilon),0\right),
	\end{equation}
	where $r_i=\kappa_i^{-1}$ is the radius of $B_i$, and
	\begin{equation}\label{qb}
		q_B(\mathbf{x})=\mathcal{N}_{\mathbf{p}_1}(\mathbf{x})-\mathcal{N}_{\mathbf{p}_2}(\mathbf{x}).
	\end{equation}
	As in \cite{emitter1}, assume $v$ is the solution of the problem
	\begin{equation}\label{equation of v}
		\begin{cases}
			\Delta v=0\ \ {\rm in}\ \Omega,\\
			v=-\mathbf{a}\cdot \nabla \mathcal{N}_{\mathbf{p}}\ {\rm on}\ \partial D_1\cup\partial D_2,\\
			\displaystyle\int_{\Omega}|\nabla v|^2{\rm d}\mathbf{x}<\infty.
		\end{cases}
	\end{equation}
	Then there exists a constant $v_0$ such that
	\begin{equation}\label{prop of v}
		v\rightarrow v_0\ \ {\rm and}\ \ |\nabla v(\mathbf{x})|=O(|\mathbf{x}|^{-2})\ {\rm as}\ |\mathbf{x}|\rightarrow\infty.
	\end{equation}
	The proof of existence and uniqueness of $v$ and (\ref{prop of v}) is standard, but \cite{emitter1} does not go into detail. We prove these results in Appendix A.
	
	Applying Green's theorem in $\Omega\cap B_R$ and then letting $R\rightarrow\infty$, we obtain
	$$\int_{\partial D_1}\partial_{\nu} v\ {\rm d}s+\int_{\partial D_2}\partial_{\nu} v\ {\rm d}s=0.$$
	Define
	$$c\coloneqq\int_{\partial D_1}\partial_{\nu} v\ {\rm d}s.$$
	Then one can check that the function $u$, defined by
	\begin{equation}\label{u}
		u(\mathbf{x})=\mathbf{a}\cdot\nabla\mathcal{N}_{\mathbf{p}}(\mathbf{x})+v(\mathbf{x})+cq(\mathbf{x})-v_0,
	\end{equation}
	is the unique solution of (\ref{equation of u}). Set
	$$Q(\mathbf{x})\coloneqq cq(\mathbf{x})-v_0\ \ {\rm and}\ \ r(\mathbf{x})\coloneqq \mathbf{a}\cdot\nabla\mathcal{N}_{\mathbf{p}}(\mathbf{x})+v(\mathbf{x}).$$
	The term $Q$ describes the field enhancement due to the interaction between two inclusions, while $r$ describes the singularity due to the existence of the emitter.

	\section{Estimates for the Components of the Decomposition}
	
	\subsection{Estimates for $q$}
	\begin{lem}
		Let $q_B$ be defined as in (\ref{qb}). Then
		\begin{equation}\label{estimate of qb}
			|\nabla q_B(\mathbf{x})|\sim\displaystyle\frac{\sqrt{\varepsilon}}{\varepsilon+x_2^2},\ \mathbf{x}\in\Omega_0.
		\end{equation}
	\end{lem}
	When $D_1$ and $D_2$ are disks with the same radius, (\ref{estimate of qb}) is listed in \cite{emitter1} without details. Here we consider different radii and give a short proof. 
	\begin{proof}
		Since $q_B=\mathcal{N}_{\mathbf{p}_1}-\mathcal{N}_{\mathbf{p}_2}$, we have
		$$|\nabla q_B(\mathbf{x})|=\frac{|\mathbf{p_1-p_2}|}{2\pi|\mathbf{x-p_1}||\mathbf{x-p_2}|}.$$
		By (\ref{p1p2}), $|\mathbf{p_1-p_2}|\sim\sqrt{\varepsilon}$. It suffices to show $|\mathbf{x-p_i}|^2\sim\varepsilon+x_2^2$. First, we have
		$$|\mathbf{x-p_i}|^2\lesssim x_1^2+\varepsilon+x_2^2\lesssim\varepsilon+x_2^2.$$
		For the reverse inequality, if $\varepsilon\lesssim x_2^2$, then
		$$\varepsilon+x_2^2\lesssim x_2^2\lesssim|\mathbf{x-p_i}|^2.$$
		If $x_2^2\lesssim\varepsilon$, then $|x_1|\lesssim\varepsilon\ll|\mathbf{p}_i|$, so
		$$\varepsilon+x_2^2\lesssim|\mathbf{x-p_i}|^2.$$
		Hence, $|\mathbf{x-p_i}|^2\sim\varepsilon+x_2^2$.
	\end{proof}
	The following lemma, proved in \cite[Proposition 3]{Ammariqqb}, relates $q$ to $q_B$.
	\begin{lem}
		We have
		\begin{equation}\label{q and qb}
			q(\mathbf{x})=a_{\varepsilon}q_B(\mathbf{x})+v_1(\mathbf{x}),\ \ \mathbf{x}\in\Omega,
		\end{equation}
		where
		\begin{equation}\label{a}
			a_{\varepsilon}\coloneqq\frac{q_2-q_1}{q_B|_{\partial B_2}-q_B|_{\partial B_1}}=1+O(\varepsilon^{\gamma/2}),
		\end{equation}
		and
		\begin{equation}\label{boundness of nabla v1}
			\Vert\nabla v_1\Vert_{L^{\infty}(\Omega)}\le C.
		\end{equation}
	\end{lem}
	By Lemmas 3.1 and 3.2, we have
	$$|\nabla q(\mathbf{x})|\lesssim\frac{\sqrt{\varepsilon}}{\varepsilon+x_2^2}+1.$$
	This is not enough for us because we want to remove the constant term. So we need more precise estimates.
	\begin{lem}
		Let $\Omega_E$ be an exterior domain and $w\in C(\overline{\Omega_E})$ be a harmonic function in $\Omega_E$ such that
		$$w(\mathbf{x})\rightarrow w_0\ {\rm as}\ |\mathbf{x}|\rightarrow\infty.$$
		Then
		$$\inf_{\partial \Omega_E}w\le w(\mathbf{x})\le\sup_{\partial \Omega_E}w.$$
	\end{lem}
	\begin{proof}
		By the maximum principle,
		$$\sup_{\Omega_E}w\le\max\{\sup_{\partial \Omega_E}w,w_0\}.$$
		If $w_0>\sup_{\partial \Omega_E}w$, then $\sup_{\Omega_E}w\le w_0$. Choose $R>0$ such that $\mathbb{R}^2\setminus\Omega_E\subset B_R$. Consider the Kelvin transformation
		$$\Phi(\mathbf{x})=\displaystyle\frac{\mathbf{x}}{|\mathbf{x}|^2}.$$
		Define
		$$\tilde{w}(\mathbf{x})\coloneqq w(\Phi(\mathbf{x})),\ \mathbf{x}\in B_{1/R}\setminus\{\mathbf{0}\}.$$
		Hence,
		$$\tilde{w}(\mathbf{x})\rightarrow w_0\ {\rm as}\ \mathbf{x}\rightarrow \mathbf{0}.$$
		We may set $\tilde{w}(\mathbf{0})=w_0$. Then $\tilde{w}$ attains its maximum at an interior point. The strong maximum principle implies $w$ is a constant, which is impossible. Consequently, $w_0\le\sup_{\partial \Omega_E}w$ and then $$w(\mathbf{x})\le\sup_{\partial \Omega_E}w.$$
		Considering $-w$ in the same way implies $w(\mathbf{x})\ge\inf_{\partial \Omega_E}w$.  
	\end{proof}
	\begin{lem}
		Let $q$ be the solution of (\ref{equation of q}). Then
		$$\Vert q\Vert_{L^{\infty}(\Omega)}\lesssim\sqrt{\varepsilon}.$$
	\end{lem}
	\begin{proof}
		Recall $q\rightarrow 0$ as $|\mathbf{x}|\rightarrow\infty$. Lemma 3.3 implies
		$$\sup_{\Omega}q=\max\{q_1,q_2\}\ge0,\ \ \ \ \inf_{\Omega}q=\min\{q_1,q_2\}\le0.$$
		By \cite[Lemma 4]{Ammariqqb}, we have $|q_2-q_1|\sim\sqrt{\varepsilon}$. Hence, $$|q_i|\lesssim\sqrt{\varepsilon},$$
		which means $\Vert q\Vert_{L^{\infty}(\Omega)}\lesssim\sqrt{\varepsilon}$.
	\end{proof}
	\begin{prop}
		For $\mathbf{x}=(x_1,x_2)\in\Omega_0$, there holds
		\begin{equation}\label{estimate of nabla q}
			|\nabla q(\mathbf{x})|\lesssim\frac{\sqrt{\varepsilon}}{\varepsilon+x_2^2}.
		\end{equation}
		In particular, there exists a constant $c_0$ such that if $|x_2|\le c_0\varepsilon^{1/4}$, then
		\begin{equation}\label{estimate of nabla qq}
			|\nabla q(\mathbf{x})|\sim\frac{\sqrt{\varepsilon}}{\varepsilon+x_2^2}.
		\end{equation}
	\end{prop}
	\begin{proof}
		For $\mathbf{x}\in\Omega_0$ with $x_1\le 0$, we claim
		$${\rm dist}(\mathbf{x},\partial D_2)\sim\varepsilon+x_2^2.$$
		In fact, we can decrease $R_0$ such that ${\rm dist}(\mathbf{x},\partial D_2)={\rm dist}(\mathbf{x},\Gamma_2)$. Then for any $\mathbf{b}=(b_1,b_2)\in\Gamma_2$,
		\begin{equation}
			\begin{aligned}
				\varepsilon+x_2^2&\lesssim \varepsilon/2+g(x_2)-x_1\\
				&\lesssim|g(x_2)-g(b_2)|+|\varepsilon/2+g(b_2)-x_1|\\
				&\lesssim|x_2-b_2|+|b_1-x_1|\lesssim|\mathbf{x-b}|.
			\end{aligned}\nonumber
		\end{equation}
		Taking the infimum implies $\varepsilon+x_2^2\lesssim{\rm dist}(\mathbf{x},\partial D_2)$. Since it is easy to obtain ${\rm dist}(\mathbf{x},\partial D_2)\lesssim\varepsilon+x_2^2$, the claim holds.
		
		By Lemma 3.4,
		$$\Vert q-q_1 \Vert_{L^{\infty}(\Omega)}\lesssim\sqrt{\varepsilon}.$$
		It then follows from the interior or boundary gradient estimates for harmonic functions that
		$$|\nabla q(\mathbf{x})|=|\nabla(q-q_1)(\mathbf{x})|\lesssim\frac{\sqrt{\varepsilon}}{{\rm dist}(\mathbf{x},\partial D_2)}\sim\frac{\sqrt{\varepsilon}}{\varepsilon+x_2^2}.$$
		When $x_1>0$, we can prove ${\rm dist}(\mathbf{x},\partial D_1)\sim\varepsilon+x_2^2$, so (\ref{estimate of nabla q}) still holds.
		
		By (\ref{estimate of qb}) and (\ref{boundness of nabla v1}), we can choose $c_0$ sufficiently small such that for all $|x_2|\le c_0\varepsilon^{1/4}$,
		$$|a_{\varepsilon}\nabla q_B(\mathbf{x})|\ge2|\nabla v_1(\mathbf{x})|.$$
		Then (\ref{estimate of nabla qq}) holds.
	\end{proof}

	\subsection{Estimates for $c$}
	We see from (\ref{equation of v}) and (\ref{u}) that
	$$u_i=cq_i-v_0,\ i=1,2.$$
	Therefore,
	\begin{equation}\label{c}
		c=\frac{u_2-u_1}{q_2-q_1}.
	\end{equation}
	By \cite[Lemma 4]{Ammariqqb},
	\begin{equation}\label{estimate of q2-q1}
		|q_2-q_1|\sim\sqrt{\varepsilon},
	\end{equation}
	so it suffices to estimate $u_2-u_1$. We first give a formula for this quantity, which is a generalization of \cite[Eq. (2.15)]{emitter1}.
	\begin{lem}
		Let $u$ be the solution of (\ref{equation of u}). Then
		\begin{equation}\label{formular of u2-u1}
			u_2-u_1=\mathbf{a}\cdot\nabla q(\mathbf{p}).
		\end{equation}
	\end{lem}
	\begin{proof}
		It follows from (\ref{equation of q}) and (\ref{u}) that
		\begin{equation}
			\begin{aligned}
				u_2-u_1&=\int_{\partial D_1\cup\partial D_2}u\partial_{\nu}q\ {\rm d}s\\
				&=\int_{\partial D_1\cup\partial D_2}\mathbf{a}\cdot\nabla\mathcal{N}_{\mathbf{p}}\partial_{\nu}q\ {\rm d}s+\int_{\partial D_1\cup\partial D_2}(v+cq-v_0)\partial_{\nu}q\ {\rm d}s.
			\end{aligned}\nonumber
		\end{equation}
		The third line of (\ref{equation of q}) and the definition of $c$ imply
		$$\int_{\partial D_1\cup\partial D_2}q\partial_{\nu}(v+cq-v_0)\ {\rm d}s=0.$$
		Since $q$ is constant on each $\partial D_i$, and $\mathbf{a}\cdot\nabla\mathcal{N}_{\mathbf{p}}$ is harmonic in $D_i$, we have
		$$\int_{\partial D_1\cup\partial D_2}q\partial_{\nu}(\mathbf{a}\cdot\nabla\mathcal{N}_{\mathbf{p}})\ {\rm d}s=0.$$
		Then Green's theorem yields
		$$u_2-u_1=-\int_{\Omega}\mathbf{a}\cdot\nabla\delta_{\mathbf{p}}\ q\ {\rm d}\mathbf{x}=\mathbf{a}\cdot\nabla q(\mathbf{p}).$$
	\end{proof}
	\begin{prop}
		Let $\mathbf{a}=(a_1,a_2)$ be a unit vector and $\mathbf{p}=(s,t)\in\Omega_0$. Then
		\begin{equation}\label{estiamte of c1}
			|c|\lesssim\frac{1}{\varepsilon+t^2}.
		\end{equation}
		More precisely, for every fixed $a_0\in(0,1)$, there exists $c_1>0$ such that if $|a_1|\ge a_0$ and $|t|\le c_1\varepsilon^{1/4}$, we have
		\begin{equation}\label{estimate of c2}
			|c|\sim\frac{|a_1|}{\varepsilon+t^2},
		\end{equation}
		and if $\mathbf{a}=(0,1)$, we have
		\begin{equation}\label{estimate of c3}
			|c|\lesssim\varepsilon^{-1/2}.
		\end{equation}
	\end{prop}
	\begin{proof}
		We see from (\ref{c}), (\ref{estimate of q2-q1}) and (\ref{formular of u2-u1}) that
		\begin{equation}\label{a calculation of c}
			|c|\sim\frac{|\mathbf{a}\cdot\nabla q(\mathbf{p})|}{\sqrt{\varepsilon}}.
		\end{equation}
		Combining this with (\ref{estimate of nabla q}) proves (\ref{estiamte of c1}).
		
		Recall $q_B=\mathcal{N}_{\mathbf{p}_1}-\mathcal{N}_{\mathbf{p}_2}$. So
		$$\nabla q_{B}(\mathbf{p})=\frac{1}{2\pi}\left(\frac{\mathbf{p}_2-\mathbf{p}}{|\mathbf{p}_2-\mathbf{p}|^2}-\frac{\mathbf{p}_1-\mathbf{p}}{|\mathbf{p}_1-\mathbf{p}|^2}\right).$$
		By (\ref{p1p2}) and a direct calculation, we have
		\begin{equation}\label{nabla qB 1}
			|\partial_{x_1}q_B(\mathbf{p})|\sim\frac{\sqrt{\varepsilon}}{\varepsilon+t^2}
		\end{equation}
		and
		\begin{equation}\label{nabla qB 2}
			|\partial_{x_2}q_B(\mathbf{p})|\lesssim\frac{\sqrt{\varepsilon}\cdot|t|\cdot(|s|+\varepsilon)}{(\varepsilon+t^2)^2}\lesssim 1.
		\end{equation}
		Thus, Lemma 3.2 implies
		$$\mathbf{a}\cdot\nabla q(\mathbf{p})=a_{\varepsilon}a_1\partial_{x_1}q_B(\mathbf{p})+O(1).$$
		If $|a_1|\ge a_0$, by (\ref{nabla qB 1}), we can choose $c_1$ sufficiently small such that (\ref{estimate of c2}) holds for $|t|\le c_1\varepsilon^{1/4}$. If $a_1=0$, then $|\mathbf{a}\cdot\nabla q(\mathbf{p})|\lesssim 1$, so (\ref{estimate of c3}) holds since $|q_2-q_1|\sim\sqrt{\varepsilon}$.
	\end{proof}
	\begin{rk}
		We can see from Proposition 3.2 that, if $\mathbf{a}=(0,1)$, the estimate of $c$ is improved. This is why the results in Theorem 1.2 are stronger than those in Theorem 1.1.
	\end{rk}

	\subsection{Estimates for $r$}
	Recall
	$$r=\mathbf{a}\cdot\nabla\mathcal{N}_{\mathbf{p}}+v.$$
	When $\mathbf{x}$ is near the emitter, we directly use 
	$$|\nabla r|\le|\nabla(\mathbf{a}\cdot\nabla\mathcal{N}_{\mathbf{p}})|+|\nabla v|$$
	to estimate $|\nabla r|$. So it suffices to estimate $|\nabla v|$.
	\begin{prop}
		Let $v$ be the solution of (\ref{equation of v}). Then
		\begin{equation}\label{estimate of L infty v gen}
			\Vert v\Vert_{L^\infty(\Omega)}\lesssim\displaystyle\frac{1}{{\rm dist}(\mathbf{p},\partial\Omega)}.
		\end{equation}
		Moreover, for every fixed $\eta\in(0,1/2)$ and every $\mathbf{x}$ satisfying $|\mathbf{x-p}|\le\eta{\rm dist}(\mathbf{p},\partial\Omega)$, we have
		\begin{equation}\label{estimate of nabla v gen}
			|\nabla v(\mathbf{x})|\lesssim\frac{1}{\eta\left[{\rm dist}(\mathbf{p},\partial\Omega)\right]^2}.
		\end{equation}
		In particular, for every fixed $\theta\in(0,1)$, if ${\rm dist}(\mathbf{p},\partial\Omega)\ge\theta\varepsilon$, then
		\begin{equation}\label{estimate of L infty v}
			\Vert v\Vert_{L^\infty(\Omega)}\lesssim\displaystyle\frac{1}{\varepsilon},
		\end{equation}
		and for every $C_1\in(0,\theta/2)$, if $|\mathbf{x-p}|\le C_1\varepsilon$, then
		\begin{equation}\label{estimate of nabla v}
			|\nabla v(\mathbf{x})|\lesssim\frac{1}{C_1\varepsilon^2},
		\end{equation}
		The constants involved in $\lesssim$ do not depend on $\eta$ and $C_1$.
	\end{prop}
	\begin{proof}
		Note that $v$ is harmonic and $v\rightarrow v_0$ as $|\mathbf{x}|\rightarrow\infty$. Lemma 3.3 implies
		$$\Vert v\Vert_{L^\infty(\Omega)}=\Vert v\Vert_{L^\infty(\partial\Omega)}\lesssim\Vert |\mathbf{x-p}|^{-1}\Vert_{L^{\infty}(\partial\Omega)}\lesssim\frac{1}{{\rm dist}(\mathbf{p},\partial\Omega)}.$$
		For any $\mathbf{x}\in\overline{B_{\eta{\rm dist}(\mathbf{p},\partial\Omega)}(\mathbf{p})}$, we have $B_{\eta{\rm dist}(\mathbf{p},\partial\Omega)/2}(\mathbf{x})\subset\Omega$. The standard interior gradient estimates for harmonic functions and (\ref{estimate of L infty v gen}) give
		$$|\nabla v(\mathbf{x})|\lesssim\frac{1}{\eta\left[{\rm dist}(\mathbf{p},\partial\Omega)\right]^2}.$$
		
		If ${\rm dist}(\mathbf{p},\partial\Omega)\ge\theta\varepsilon$, for every $C_1\in(0,\theta/2)$, we have $C_1/\theta\in(0,1/2)$. Then (\ref{estimate of L infty v}) and (\ref{estimate of nabla v}) follow from (\ref{estimate of L infty v gen}) and (\ref{estimate of nabla v gen}) by setting $\eta=C_1/\theta$.
	\end{proof}
	When $\mathbf{x}$ is not near the emitter, using the triangle inequality to estimate $|\nabla r|$ is not enough, and we need a more detailed analysis. In \cite{emitter1}, the estimate of $|\nabla r|$ can be proved by a Kelvin transformation since the inclusions are circular. As for the general case, we construct the harmonic barrier to control the $L^{\infty}$ norm of $r$ locally. Then standard gradient estimates for harmonic functions give the estimates of $|\nabla r|$. The estimates are not the same as those in \cite{emitter1}, and they may not be sharp, but they are sufficient for the proofs of the main theorems.
	\begin{prop}
		Fix $M>0,\ \theta\in(0,1)$. Let $\mathbf{a}$ be a unit vector and let $\mathbf{p}=(s,t)\in\Omega_0$ satisfy $|t|\le M\sqrt{\varepsilon}$ and ${\rm dist}(\mathbf{p},\partial\Omega)\ge\theta\varepsilon$. Fix $C_1$ sufficiently small such that $B_{2C_1\varepsilon}(\mathbf{p})\subset\Omega_0$. Then there exist positive constants $A,\ A_1$ such that the following estimates hold for every $\mathbf{x}\in\Omega_0$:
		\begin{enumerate}[(i)]
			\item If $C_1\varepsilon<|\mathbf{x-p}|\le 2\sqrt{\varepsilon}$, then
			\begin{equation}\label{estimate of nabla r}
				|\nabla r(\mathbf{x})|\lesssim\frac{1}{|\mathbf{x-p}|^2}\exp{\left(-\frac{A|\mathbf{x-p}|}{\varepsilon}\right)}.
			\end{equation}
			\item If $|\mathbf{x-p}|>2\sqrt{\varepsilon}$, then
			\begin{equation}\label{estimate of nabla rr}
				|\nabla r(\mathbf{x})|\lesssim\frac{1}{\varepsilon}\exp{\left(-\frac{A_1}{\sqrt{\varepsilon}}\right)}.
			\end{equation}
		\end{enumerate}
		The constants $A,\ A_1$, and the implicit constants may additionally depend on $C_1$.
	\end{prop}
	\begin{proof}
		Since ${\rm dist}(\mathbf{p},\partial\Omega)\ge\theta\varepsilon$, the definition of $r$ and (\ref{estimate of L infty v gen}) give
		\begin{equation}\label{estiamte of r}
			|r(\mathbf{x})|\lesssim\frac{1}{|\mathbf{x-p}|}+\frac{1}{{\rm dist}(\mathbf{p},\partial\Omega)}\lesssim\frac{1}{\varepsilon}
		\end{equation}
		when $|\mathbf{x-p}|>C_1\varepsilon/2$.
		
		We first prove (\ref{estimate of nabla r}). Fix a point $\mathbf{x}_0=(x_{01},x_{02})\in\Omega_0$ satisfying $$C_1\varepsilon<|\mathbf{x_0-p}|\le 2\sqrt{\varepsilon}.$$
		
		\textbf{Case 1:} $|x_{02}-t|>4C_1\varepsilon$.
		
		We only consider $x_{02}-t>4C_1\varepsilon$. The other case can be proved in the same way. Define
		$$S\coloneqq\left\{(x_1,x_2)\in\Omega_0:x_{02}-C_1\varepsilon<x_2<x_{02}+C_1\varepsilon\right\},$$
		and
		$$S_1\coloneqq\left\{(x_1,x_2)\in\Omega_0:\frac{x_{02}+t}{2}<x_2<\frac{x_{02}+t}{2}+x_{02}-t\right\}.$$
		Then $S\subset S_1$ and $\mathbf{p}\notin\bar{S_1}$.
		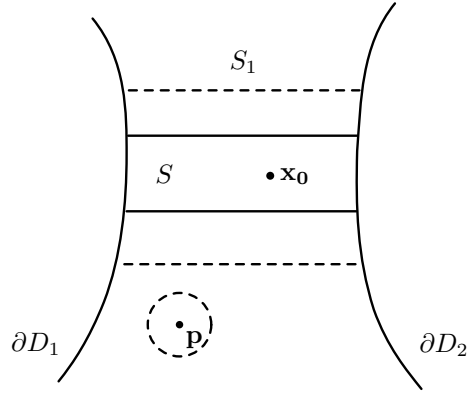
\begin{figure}[htbp]
			\centering
			\begin{tikzpicture}[
				>=Stealth,
				line width=0.9pt,
				line cap=round,
				line join=round
				]
				\draw
				(-2.0, 2.6)
				.. controls (-1.55, 2.05) and (-1.55, 1.35) .. (-1.55, 0.8)
				.. controls (-1.55, 0.15) and (-1.62,-0.55) .. (-1.85,-1.15)
				.. controls (-2.05,-1.65) and (-2.25,-1.95) .. (-2.45,-2.20);
				\draw
				( 2.0, 2.8)
				.. controls ( 1.55, 2.25) and ( 1.55, 1.50) .. ( 1.50, 0.85)
				.. controls ( 1.48, 0.10) and ( 1.50,-0.55) .. ( 1.72,-1.25)
				.. controls ( 1.90,-1.75) and ( 2.15,-2.05) .. ( 2.35,-2.30);
				\draw[dashed]
				(-1.52,1.65) -- (1.53,1.65);
				\draw
				(-1.53,1.05) -- (1.50,1.05);
				\draw
				(-1.55,0.05) -- (1.49,0.05);
				\draw[dashed]
				(-1.58,-0.65) -- (1.55,-0.65);
				\node at (0.00,2.02) {$S_1$};
				\node at (-1.05,0.55) {$S$};
				\fill (0.35,0.52) circle (1.5pt);
				\node[right] at (0.35,0.52) {$\mathbf{x_0}$};
				\draw[dashed] (-0.85,-1.45) circle (0.42);
				\fill (-0.85,-1.45) circle (1.4pt);
				\node[below right] at (-0.90,-1.40) {$\mathbf{p}$};
				\node at (-2.75,-1.70) {$\partial D_1$};
				\node at ( 2.65,-1.72) {$\partial D_2$};
			\end{tikzpicture}
			\caption{A configuration of $S$ and $S_1$.}
		\end{figure}
	
		Define a harmonic function
		$$h(x_1,x_2)\coloneqq\frac{C_0\varepsilon}{(x_{02}-t)^2}\exp{\frac{\lambda (t-x_{02})}{4\varepsilon}}\cos{\frac{\lambda x_1}{\varepsilon}}\cosh{\frac{\lambda(x_2-x_{02})}{\varepsilon}},$$
		where $C_0$ and $\lambda$ are constants to be determined. For $\mathbf{x}=(x_1,x_2)\in S_1$, we have
		$$|x_1|\lesssim\varepsilon+x_2^2\lesssim\varepsilon+x_{02}^2+t^2\lesssim\varepsilon,$$
		so we can choose a small constant $\lambda>0$ such that
		$$\cos{\frac{\lambda x_1}{\varepsilon}}>\frac{1}{2}.$$

		On $\partial S_1\cap(\partial D_1\cup\partial D_2)$, since $r=0$, we have $h\ge |r|$. On $\partial S_1\backslash(\partial D_1\cup\partial D_2)$, we have
		$$|x_2-x_{02}|=\frac{x_{02}-t}{2}.$$
		Hence,
		\begin{equation}\label{estimate of h}
			h(\mathbf{x})\ge\frac{C_0\varepsilon}{4(x_{02}-t)^2}\exp{\frac{\lambda (x_{02}-t)}{4\varepsilon}}.
		\end{equation}
		Since $(x_{02}-t)/\varepsilon>4C_1$, the quantity
		$$\frac{\varepsilon^2}{(x_{02}-t)^2}\exp{\frac{\lambda (x_{02}-t)}{4\varepsilon}}$$
		has a positive infimum depending only on $C_1$. It follows from (\ref{estiamte of r}) and (\ref{estimate of h}) that we can choose $C_0$ sufficiently large such that  $h(\mathbf{x})\ge|r(\mathbf{x})|$ on $\partial S_1\backslash(\partial D_1\cup\partial D_2)$. The maximum principle implies $$|r(\mathbf{x})|\le h(\mathbf{x}),\ \ \mathbf{x}\in S_1.$$
		For $\mathbf{x}\in S$, since $|x_2-x_{02}|\lesssim\varepsilon$, we have
		$$\cosh{\frac{\lambda(x_2-x_{02})}{\varepsilon}}\lesssim 1.$$
		Therefore,
		\begin{equation}\label{estimate of r improved}
			\Vert r \Vert_{L^{\infty}(S)}\lesssim\frac{\varepsilon}{(x_{02}-t)^2}\exp{\frac{\lambda (t-x_{02})}{4\varepsilon}}.
		\end{equation}
		Since the length and width of $S$ are of order $\varepsilon$, the interior or boundary gradient estimates for harmonic functions imply
		\begin{equation}\label{estimate of nabla rx0}
			|\nabla r(\mathbf{x_0})|\lesssim\frac{1}{\varepsilon}\Vert r \Vert_{L^{\infty}(S)}\lesssim\frac{1}{(x_{02}-t)^2}\exp{\frac{\lambda (t-x_{02})}{4\varepsilon}}.
		\end{equation}
		Moreover, the assumptions on $\mathbf{x_0}$ lead to
		\begin{equation}\label{d}
			x_{02}-t\sim |\mathbf{x_0-p}|.
		\end{equation}
		Combining this with (\ref{estimate of nabla rx0}) proves (\ref{estimate of nabla r}).
		
		\textbf{Case 2:} $|x_{02}-t|\le 4C_1\varepsilon$.
		
		Since $|\mathbf{x_0-p}|>C_1\varepsilon$, the distance from $\mathbf{x_0}$ to $\partial B_{C_1\varepsilon/2}(\mathbf{p})$ is at least $C_1\varepsilon/2$. Using  (\ref{estiamte of r}) and the interior or boundary gradient estimates for harmonic functions at a scale $\varepsilon$ yields
		$$|\nabla r(\mathbf{x_0})|\lesssim \frac{1}{\varepsilon^2}.$$
		Moreover, it is easy to check
		$$\varepsilon\sim|\mathbf{x_0-p}|.$$
		Hence, (\ref{estimate of nabla r}) holds.
				
		Now we prove (\ref{estimate of nabla rr}). Since $|\mathbf{x-p}|>2\sqrt{\varepsilon}$, the estimate $|x_1|\lesssim\varepsilon$ may not hold. So we cannot use the harmonic barrier directly. Define
		$$\widetilde{S}\coloneqq\{\mathbf{x}\in\Omega:|\mathbf{x-p}|>2\sqrt{\varepsilon}\},$$
		$$\widetilde{S_1}\coloneqq\{\mathbf{x}\in\Omega:|\mathbf{x-p}|>\sqrt{\varepsilon}\}.$$
		Then $\widetilde{S}\subset\widetilde{S_1}$. For any $\mathbf{x_0}\in\partial\widetilde{S_1}\backslash(\partial D_1\cup\partial D_2)$, we have
		$$|x_{02}-t|\sim\sqrt{\varepsilon}>4C_1\varepsilon.$$
		Hence, we can use the harmonic barrier argument to obtain
		$$|r(\mathbf{x_0})|\lesssim\frac{\varepsilon}{|\mathbf{x_0-p}|^2}\exp{\left(-\frac{\lambda|\mathbf{x_0-p}|}{4\varepsilon}\right)}.$$
		Since the points on $\partial\widetilde{S_1}\backslash(\partial D_1\cup\partial D_2)$ satisfy $|\mathbf{x_0-p}|=\sqrt{\varepsilon}$, and $r$ vanishes on $\partial\Omega$, we have
		$$\Vert r \Vert_{L^\infty(\partial \widetilde{S_1})}\lesssim\exp{\left(-\frac{A_1}{\sqrt{\varepsilon}}\right)}.$$
		Moreover, $r(\mathbf{x})\rightarrow v_0$ as $|\mathbf{x}|\rightarrow\infty$. Then Lemma 3.3 yields
		$$\Vert r \Vert_{L^{\infty}(\widetilde{S_1})}\lesssim\exp{\left(-\frac{A_1}{\sqrt{\varepsilon}}\right)}.$$
		For every $\mathbf{x}\in \widetilde{S}$, we have $\varepsilon+x_2^2\ge\varepsilon$, and the distance from $\mathbf{x}$ to $\partial \widetilde{S_1}\setminus(\partial D_1\cup\partial D_2)$ is at least $\sqrt{\varepsilon}$. So we can use the interior or boundary gradient estimates for harmonic functions at a scale $\varepsilon$. Then we obtain (\ref{estimate of nabla rr}).
	\end{proof}

	\section{Proofs of Theorems 1.1 and 1.2}
	\renewcommand{\proofname}{Proof of Theorem 1.1.\nopunct}
	\begin{proof}
		Recall $u=Q+r$, where
		$$Q(\mathbf{x})=cq(\mathbf{x})-v_0,\ \ r(\mathbf{x})=\mathbf{a}\cdot\nabla\mathcal{N}_{\mathbf{p}}(\mathbf{x})+v(\mathbf{x}).$$
		We just need to combine the previous results together.
		
		\textbf{Near the emitter.}
		If $|\mathbf{x-p}|\le C_1\varepsilon$, then (\ref{estimate of nabla qq}) and (\ref{estimate of c2}) give
		\begin{equation}\label{estimate of nabla Q}
			|\nabla Q(\mathbf{x})|\sim\frac{|a_1|}{\sqrt{\varepsilon}(\varepsilon+x_2^2)}.
		\end{equation}
		A direct calculation gives
		\begin{equation}\label{estimate of aNp}
			|\nabla(\mathbf{a}\cdot\nabla\mathcal{N}_{\mathbf{p}})(\mathbf{x})|=\frac{1}{2\pi|\mathbf{x-p}|^2}\ge\frac{1}{2\pi C_1^2\varepsilon^2}.
		\end{equation}
		Recall
		\begin{equation}\label{estimate of nabla v Section 4}
			|\nabla v|\lesssim\frac{1}{C_1\varepsilon^2},
		\end{equation}
		so we can choose $C_1$ sufficiently small such that $|\nabla(\mathbf{a}\cdot\nabla\mathcal{N}_{\mathbf{p}})|$ is the leading term of $|\nabla u|$. This proves (\ref{thm1}).
		
		\textbf{Far from the emitter.}
		Set
		$$C_2=\frac{2}{A},$$
		where $A$ is the constant in Proposition 3.4.
		
		For (\ref{thm2}), since $|x_2|\le c_0\varepsilon^{1/4}$, (\ref{estimate of nabla Q}) still holds. If $|\mathbf{x-p}|\le 2\sqrt{\varepsilon}$, since $|\mathbf{x-p}|\ge C_2\varepsilon|\log{\varepsilon}|$, (\ref{estimate of nabla r}) implies
		\begin{equation}\label{estimate of nabla r in Section 4}
			|\nabla r(\mathbf{x})|\lesssim\frac{1}{|\log{\varepsilon}|^2}\ll\frac{1}{\sqrt{\varepsilon}(\varepsilon+x_2^2)}.
		\end{equation}
		If $|\mathbf{x-p}|\ge 2\sqrt{\varepsilon}$, (\ref{estimate of nabla rr}) implies
		\begin{equation}\label{estiamte of nabla rr in Section 4}
			|\nabla r(\mathbf{x})|\lesssim\frac{1}{\varepsilon}\exp{\left(-\frac{A_1}{\sqrt{\varepsilon}}\right)}\ll\frac{1}{\sqrt{\varepsilon}(\varepsilon+x_2^2)},
		\end{equation}
		Combining (\ref{estimate of nabla Q}) with (\ref{estimate of nabla r in Section 4}) and (\ref{estiamte of nabla rr in Section 4}) proves (\ref{thm2}).
		
		As for $(\ref{thm3})$, since $|x_2|>c_0\varepsilon^{1/4}$, we see from (\ref{estimate of nabla q}) and (\ref{estiamte of c1}) that
		\begin{equation}\label{estimate of QQ}
			|\nabla Q(\mathbf{x})|\lesssim\frac{1}{\sqrt{\varepsilon}(\varepsilon+x_2^2)}.
		\end{equation}
		Combining this with (\ref{estiamte of nabla rr in Section 4}) proves (\ref{thm3}).
		
		\textbf{Transition region.}
		(\ref{thm4}) can be proved by (\ref{estimate of QQ}) and (\ref{estimate of nabla r}).
	\end{proof}
	Now we consider Theorem 1.2.
	\renewcommand{\proofname}{Proof of Theorem 1.2.\nopunct}
	\begin{proof}
		By Proposition 3.2, $|c|\lesssim\varepsilon^{-1/2}$. Therefore,
		\begin{equation}\label{estimate of nabla Q 01}
			|\nabla Q(\mathbf{x})|\lesssim\frac{1}{\varepsilon+x_2^2}.
		\end{equation}
	
		\textbf{Near the emitter.}
		By (\ref{estimate of aNp}), (\ref{estimate of nabla v Section 4}) and (\ref{estimate of nabla Q 01}), we can choose $C_1$ sufficiently small such that $|\nabla(\mathbf{a}\cdot\nabla\mathcal{N}_{\mathbf{p}})|$ is the leading term of $|\nabla u|$.
		 
		\textbf{Far from the emitter.}
		As in the proof of Theorem 1.1, we can choose $C_2$ such that
		$$|\nabla r(\mathbf{x})|\ll\frac{1}{\varepsilon+x_2^2}.$$
		Combining this with (\ref{estimate of nabla Q 01}) implies (\ref{thm22}).
		
		\textbf{Transition region.}
		(\ref{thm33}) can be proved by (\ref{estimate of nabla Q 01}) and (\ref{estimate of nabla r}).
	\end{proof}
	\begin{rk}
		Under the assumptions of Theorem 1.2, if $\mathbf{p}=(0,t)$ and $D_1\cup D_2$ is symmetric with respect to the $x_2$-axis, then $q$ is odd in $x_1$. For $\mathbf{a}=(0,1)$, we have
		$$\mathbf{a}\cdot\nabla q(\mathbf{p})=0,$$
		so $c=0$. Hence,
		$$\nabla u=\nabla r.$$
		The estimates in Theorem 1.2 can therefore be improved. For example, if $|t|\le\sqrt{\varepsilon}$ and $\mathbf{x}\in\Omega_0$ satisfies $|\mathbf{x}|>2\sqrt{\varepsilon}$, then $|\mathbf{x-p}|\ge\sqrt{\varepsilon}$. It follows from (\ref{estimate of nabla r}) and (\ref{estimate of nabla rr}) that
		$$|\nabla u(\mathbf{x})|\lesssim\frac{1}{\varepsilon}\exp{\left(-\frac{A*}{\sqrt{\varepsilon}}\right)}\lesssim\exp{\left(-\frac{A*}{2\sqrt{\varepsilon}}\right)},$$
		where $A*>0$ is a constant independent of $\varepsilon$. This shows $|\nabla u|$ decays exponentially as $\varepsilon$ tends to $0$, which is the same as in \cite{emitter1}.
	\end{rk}
	If the condition ${\rm dist}(\mathbf{p},\partial\Omega)\ge\theta\varepsilon$ does not hold, Proposition 3.3 shows that $|v|$ and $|\nabla v|$ can be larger, and we may need more precise methods to estimate $|\nabla v|$ and $|\nabla r|$. However, we can still get the estimate for the points near the emitter.
	\begin{prop}
		Fix $M>0$. Then there exist constants $\eta\in(0,1/2),\ \varepsilon_0>0$ such that for every $0<\varepsilon\le\varepsilon_0$, every $\mathbf{p}=(s,t)\in\Omega_0$ satisfying $|t|\le M\sqrt{\varepsilon}$, every unit vector $\mathbf{a}$, and every $\mathbf{x}\in\Omega_0$ satisfying $0<|\mathbf{x-p}|\le\eta{\rm dist}(\mathbf{p},\partial\Omega)$,
		we have
		$$|\nabla u(\mathbf{x})|\sim\frac{1}{|\mathbf{x-p}|^2}.$$
	\end{prop}
	\renewcommand{\proofname}{Proof.\nopunct}
	\begin{proof}
		By (\ref{estimate of nabla q}) and (\ref{estiamte of c1}), we have
		$$|\nabla Q|\lesssim\varepsilon^{-3/2}\lesssim\left[{\rm dist}(\mathbf{p},\partial\Omega)\right]^{-3/2},$$
		while (\ref{estimate of nabla v gen}) implies
		$$|\nabla v(\mathbf{x})|\lesssim\frac{1}{\eta\left[{\rm dist}(\mathbf{p},\partial\Omega)\right]^2}.$$
		Recall
		$$|\nabla(\mathbf{a}\cdot\nabla\mathcal{N}_{\mathbf{p}})|=\frac{1}{2\pi|\mathbf{x-p}|^2}\ge\frac{1}{2\pi\eta^2\left[{\rm dist}(\mathbf{p},\partial\Omega)\right]^2}.$$
		Then we can choose $\eta$ sufficiently small such that $|\nabla(\mathbf{a}\cdot\nabla\mathcal{N}_{\mathbf{p}})|$ is the leading term of $|\nabla u|$.
	\end{proof}

	\section{Effect of Emitter Location on the Field}
	\renewcommand{\proofname}{Proof of Theorem 1.3.\nopunct}
	\begin{proof}
		By (\ref{estiamte of c1}),
		$$|c|\lesssim\frac{1}{\varepsilon+t^2}.$$
		For $\mathbf{x}\in L$, (\ref{estimate of nabla qq}) gives
		$$|\nabla q(\mathbf{x})|\sim\frac{1}{\sqrt{\varepsilon}}.$$
		Hence,
		\begin{equation}\label{estimate of Q 5}
			|\nabla Q(\mathbf{x})|=|c||\nabla q(\mathbf{x})|\lesssim\frac{1}{\sqrt{\varepsilon}(\varepsilon+t^2)}.
		\end{equation}
		Let
		$$\hat{S}\coloneqq\{\mathbf{x}\in\Omega_0:\ |x_2|<\varepsilon\},$$
		$$\hat{S_1}\coloneqq\{\mathbf{x}\in\Omega_0:\ |x_2|<\sqrt{\varepsilon}\}.$$
		Note that $|t|>4\sqrt{\varepsilon}$, so $\mathbf{p}\notin\overline{\hat{S_1}}$. Hence $r$ is harmonic in $\hat{S_1}$. Define a harmonic function
		$$h_0(x_1,x_2)\coloneqq \frac{\bar{C}}{{\rm dist}(\mathbf{p},\partial\Omega)}\exp{\left(-\frac{\mu}{\sqrt{\varepsilon}}\right)}\cos{\frac{\mu x_1}{\varepsilon}}\cosh{\frac{\mu x_2}{\varepsilon}},$$
		where $\bar{C}>0$ and $\mu>0$ are constants to be determined.

		On $\partial\hat{S_1}\setminus(\partial D_1\cup\partial D_2)$, we have $|x_2|=\sqrt{\varepsilon}$, so
		$$\exp{\left(-\frac{\mu}{\sqrt{\varepsilon}}\right)}\cosh{\frac{\mu x_2}{\varepsilon}}>\frac{1}{2}.$$
		Choose $\mu$ such that for all $\mathbf{x}\in\hat{S_1}$,
		$$\cos{\frac{\mu x_1}{\varepsilon}}>\frac{1}{2}.$$
		Then
		$$h_0(x_1,x_2)\ge\frac{\bar{C}}{4{\rm dist}(\mathbf{p},\partial\Omega)},\ \ |x_2|=\sqrt{\varepsilon}.$$
		Since
		$${\rm dist}(\mathbf{p},\partial\Omega)\lesssim\varepsilon+t^2\lesssim |t|,$$
		we obtain
		$$|r(\mathbf{x})|\lesssim||t|-\sqrt{\varepsilon}|^{-1}+\left[{\rm dist}(\mathbf{p},\partial\Omega)\right]^{-1}\lesssim\left[{\rm dist}(\mathbf{p},\partial\Omega)\right]^{-1}.$$
		Therefore, we can choose $\bar{C}$ sufficiently large such that $|r|\le h_0$ on $\partial\hat{S_1}\setminus(\partial D_1\cup\partial D_2)$.
		
		On $\partial\hat{S_1}\cap(\partial D_1\cup\partial D_2)$, since $r=0$, we have $|r|\le h_0$. So $|r|\le h_0$ on $\partial\hat{S_1}$. Then the maximum principle  implies
		$$|r(\mathbf{x})|\le h_0(\mathbf{x}),\ \ \mathbf{x}\in\hat{S_1}.$$
		In particular, in the smaller strip $\hat{S}$, we have 
		$$\Vert r\Vert_{L^{\infty}(\hat{S})}\lesssim\frac{1}{{\rm dist}(\mathbf{p},\partial\Omega)}\exp{\left(-\frac{A}{\sqrt{\varepsilon}}\right)}.$$
		The interior or boundary gradient estimates for harmonic functions at a scale $\varepsilon$ imply
		\begin{equation}\label{5 estimate of nabla r}
			|\nabla r(\mathbf{x})|\lesssim \frac{1}{\varepsilon{\rm dist}(\mathbf{p},\partial\Omega)}\exp{\left(-\frac{A}{\sqrt{\varepsilon}}\right)},\ \mathbf{x}\in L.
		\end{equation}
		Then (\ref{estimate of Q 5}) and (\ref{5 estimate of nabla r}) imply (\ref{location estimate1}).
		
		If ${\rm dist}(\mathbf{p},\partial\Omega)\ge\theta\varepsilon$, then
		$$|\nabla r|\lesssim\varepsilon^{-2}\exp{\left(-\frac{A}{\sqrt{\varepsilon}}\right)}\ll\frac{1}{\sqrt{\varepsilon}(\varepsilon+t^2)}.$$
		Hence, (\ref{location estimate2}) holds.
	\end{proof}
	\begin{rk}
		Under the assumptions of Theorem 1.3, suppose $D_1$ and $D_2$ are disks with the same radius $\delta_0$ and $\mathbf{p}=(0,t)$. Then $q=q_B$, and $q_B$ is odd in $x_1$. Therefore
		$$|c|\sim\frac{|\mathbf{a}\cdot\nabla q_B(\mathbf{p})|}{\sqrt{\varepsilon}}\sim\frac{|a_1\partial_{x_1} q_B(\mathbf{p})|}{\sqrt{\varepsilon}}\sim\frac{1}{\varepsilon+t^2}.$$
		Thus,
		$$|\nabla Q|\sim\frac{1}{\sqrt{\varepsilon}(\varepsilon+t^2)},\ \mathbf{x}\in L.$$
		Since
		$${\rm dist}(\mathbf{p},\partial\Omega)=\sqrt{(\delta_0+\varepsilon/2)^2+t^2}-\delta_0\sim\varepsilon+t^2,$$
		we see from (\ref{5 estimate of nabla r}) that
		$$|\nabla r|\lesssim\varepsilon^{-2}\exp{\left(-\frac{A}{\sqrt{\varepsilon}}\right)}\ll\frac{1}{\sqrt{\varepsilon}(\varepsilon+t^2)}.$$
		Consequently,
		$$|\nabla u(\mathbf{x})|\sim\frac{1}{\sqrt{\varepsilon}(\varepsilon+t^2)},\ \mathbf{x}\in L.$$
		This shows the estimate (\ref{location estimate2}) is optimal.
	\end{rk}

	\appendix
	
	\section{Exterior Dirichlet Problem in Section 2}
	
	\begin{prop}
		Under the assumptions in Section 1 and Section 2, problem (\ref{equation of v}) has a unique solution $v$. Moreover, $v$ satisfies (\ref{prop of v}).
	\end{prop}
	\renewcommand{\proofname}{Proof.\nopunct}
	\begin{proof}
		We use the Lax--Milgram theorem to clarify the existence and uniqueness of $v$. Choose $\phi\in C_c^{\infty}(\mathbb{R}^2)$ such that
		$$\phi=-\mathbf{a}\cdot\nabla\mathcal{N}_{\mathbf{p}}\ \ {\rm on}\ \partial D_1\cup\partial D_2.$$
		It suffices to prove the existence and uniqueness of the following problem:
		\begin{equation}\label{equation of v aux}
			\begin{cases}
				\Delta\bar{v}=-\Delta\phi\ \ {\rm in}\ \Omega,\\
				\bar{v}=0\ {\rm on}\ \partial D_1\cup\partial D_2,\\
				\displaystyle\int_{\Omega}|\nabla\bar{v}|^2{\rm d}\mathbf{x}<\infty.
			\end{cases}
		\end{equation}
		Define
		$$H\coloneqq\{\psi:\psi\in H^1(\Omega\cap B_R)\ {\rm for}\ {\rm all}\ {\rm large}\ R,\ \nabla \psi\in L^2(\Omega),\ \psi=0\ {\rm on}\ \partial D_1\cup\partial D_2\}$$
		and consider the inner product
		$$(\psi_1,\psi_2)\coloneqq\int_{\Omega}\nabla\psi_1\cdot\nabla\psi_2\ {\rm d}\mathbf{x}.$$
		We can check $H$ is a Hilbert space. Consider the bilinear map
		$$B[\psi_1,\psi_2]\coloneqq(\psi_1,\psi_2),\ \ \psi_1,\ \psi_2\in H,$$
		which is bounded and coercive. Then for the bounded linear functional
		$$\mathcal{F}(\psi)\coloneqq-(\phi,\psi),\ \ \psi\in H,$$
		the Lax--Milgram theorem implies that there exists a unique $\bar{v}\in H$ such that
		$$B[\bar{v},\psi]=\mathcal{F}(\psi),\ {\rm for}\ {\rm every}\ \psi\in H.$$
		Then we get a unique weak solution of (\ref{equation of v aux}). The regularity theory shows that such a solution is classical. Therefore, the existence and uniqueness of $v$ are proved.
		
		We next prove (\ref{prop of v}). Choose $R>0$ such that $\mathbb{R}^2\setminus\Omega\subset B_R$, and consider the Kelvin transformation
		$$\Phi(\mathbf{x})=\frac{\mathbf{x}}{|\mathbf{x}|^2}.$$
		Define
		$$\tilde{v}(\mathbf{x})\coloneqq v(\Phi(\mathbf{x})),\ \ \mathbf{x}\in B_{1/R}\setminus\{\mathbf{0}\}.$$
		One can check
		$$\int_{B_{1/R}\setminus\{\mathbf{0}\}}|\nabla\tilde{v}|^2\ {\rm d}\mathbf{x}=\int_{\mathbb{R}^2\setminus B_R}|\nabla v|^2\ {\rm d}\mathbf{x}<\infty.$$
		We use complex analysis methods to show the origin is a removable singularity of $\tilde{v}$. Denote
		$$z=x_1+ix_2.$$
		Since $\tilde{v}$ is harmonic in $B_{1/R}\setminus\{\mathbf{0}\}$, we can define a holomorphic function
		$$F(z)=\partial_{x_1}\tilde{v}-i\partial_{x_2}\tilde{v},\ \ z\in B_{1/R\setminus\{\mathbf{0}\}}.$$
		Consider the Laurent expansion
		$$F(z)=\sum_{n=-\infty}^{\infty}\alpha_nz^n.$$
		A direct calculation implies
		$$\int_{B_{1/R}\setminus\{\mathbf{0}\}}|\nabla\tilde{v}|^2\ {\rm d}\mathbf{x}=2\pi\sum_{n=-\infty}^{\infty}|\alpha_n|^2\int_0^{1/R}y^{2n+1}\ {\rm d}y.$$
		The finite-energy condition then implies $\alpha_n=0$ for $n<0$. Hence, $\mathbf{0}$ is a removable singularity of $F$, and consequently also of $\nabla\tilde{v}$. Then $|\nabla\tilde{v}|$ is bounded near the origin. By the chain rule, we obtain
		$$|\nabla v(\mathbf{x})|=O(|\mathbf{x}|^{-2})\ \ {\rm as}\ |\mathbf{x}|\rightarrow\infty.$$
		Moreover, by the boundedness of $|\nabla\tilde{v}|$ near the origin, $\mathbf{0}$ is also a removable singularity of $\tilde{v}$, so there exists $v_0$ such that
		$$v\rightarrow v_0\ \ {\rm as}\ |\mathbf{x}|\rightarrow\infty.$$
	\end{proof}

\end{document}